\documentclass[a4paper,11pt,twoside]{article}

\usepackage{amsmath,amssymb,amsthm,mathrsfs}

\usepackage{hyperref}

\numberwithin{equation}{section}

\newtheorem{Thm}{Theorem}[section]
\newtheorem{Prop}[Thm]{Proposition}
\newtheorem{Lem}[Thm]{Lemma}
\theoremstyle{definition}
\newtheorem{Expl}[Thm]{Example}

\newcommand{\Q}{\mathbb{Q}}
\newcommand{\R}{\mathbb{R}}
\newcommand{\Z}{\mathbb{Z}}

\newcommand{\E}{\operatorname{E}}

\newcommand{\cR}{\mathscr{R}}

\newcommand{\bx}{\mathbf{x}}

\renewcommand{\b}{\bar}
\newcommand{\bary}{\operatorname{bar}}
\newcommand{\CAT}{\operatorname{CAT}}
\newcommand{\del}{\delta}
\newcommand{\Del}{\Delta}
\newcommand{\eps}{\varepsilon}
\newcommand{\es}{\emptyset}
\newcommand{\Gam}{\Gamma}
\newcommand{\gam}{\gamma}
\newcommand{\im}{\operatorname{im}}
\newcommand{\Isom}{\operatorname{Isom}}
\newcommand{\lam}{\lambda}
\newcommand{\Min}{\operatorname{Min}}
\newcommand{\ol}{\overline}
\newcommand{\olC}{\,\overline{\!C}}
\newcommand{\olf}{\,\overline{\!f}}
\renewcommand{\phi}{\varphi}
\renewcommand{\rho}{\varrho}
\newcommand{\sig}{\sigma}
\newcommand{\sm}{\setminus}
\newcommand{\sub}{\subset}

\title{Flats in spaces with convex geodesic bicombings}

\author{Dominic Descombes \& 
Urs Lang\thanks{Research supported by the Swiss National Science Foundation.}}

\date{March 2, 2016}

\begin{document}


\maketitle

\begin{abstract}
In spaces of nonpositive curvature the existence of isometrically embedded 
flat (hyper)planes is often granted by apparently weaker conditions on large 
scales. We show that some such results remain valid for metric spaces
with non-unique geodesic segments under suitable convexity assumptions 
on the distance function along distinguished geodesics. 
The discussion includes, among other things, the Flat Torus Theorem and 
Gromov's hyperbolicity criterion referring to embedded planes. 
This generalizes results of Bowditch for Busemann spaces.
\end{abstract}
 

\section{Introduction}

The geometry of spaces of global nonpositive curvature is largely dominated by 
the convexity of the distance function. Thus a considerable part of 
the theory of $\CAT(0)$ spaces~\cite{Bal, BriH} carries over to Busemann 
spaces~\cite{Bus, Pap} (defined by the property that $d \circ (\sig_1,\sig_2)$
is convex for any pair of constant speed geodesics $\sig_1,\sig_2$ 
parametrized on the same interval). 
However, this larger class of spaces has the defect of not being preserved 
under limit processes. For example, among normed real vector spaces, 
exactly those with strictly convex norm satisfy the Busemann property, 
and a sequence of such norms on $\R^n$, say, may converge to a non-strictly 
convex norm. This motivates the study of an even weaker 
notion of nonpositive curvature that dispenses with the uniqueness of 
geodesics but retains the convexity condition for a suitable selection of 
geodesics (compare Sect.~10 in~\cite{Kle}). In any normed space,
the affine segments $t \mapsto (1-t)x + ty$ ($t \in [0,1]$) provide 
such a choice. In particular, the relaxed condition carries the potential for 
simultaneous generalizations of results for nonpositively curved and 
Banach spaces. Another reason for this investigation is that $l_1$- and 
$l_\infty$-type metrics have been put in use in geometric group theory; 
see, for example, \cite{BehDS, Bow2, ChaDH, Lan}. 
The recent paper~\cite{KapL} shows that symmetric spaces of noncompact 
type possess natural, non-strictly convex Finsler metrics adapted to the 
geometry of Weyl chambers and pertinent to the dynamics at infinity.

In a previous article we initiated a systematic study of
spaces of weak global nonpositive curvature as described above, with the main 
objective of providing geometric models of this type for hyperbolic groups;
see~\cite{DesL} and in particular Theorem~1.3 therein. 
The purpose of the present paper is to carry the analogy with $\CAT(0)$ and 
Busemann spaces further with regard to existence results for flat subspaces. 
Here, for a metric space $X = (X,d)$, a map 
$\sig \colon X \times X \times [0,1] \to X$ will be simply called a 
{\em bicombing} if the respective family of maps 
$\sig_{xy} := \sig(x,y,\cdot) \colon [0,1] \to X$ satisfies the following 
three properties:
\begin{enumerate}
\item[\rm (i)]
$\sig_{xy}$ is a geodesic from $x$ to $y$, that is, 
$\sig_{xy}(0) = x$, $\sig_{xy}(1) = y$, and 
$d(\sig_{xy}(t),\sig_{xy}(t')) = |t-t'|\,d(x,y)$ for $t,t' \in [0,1]$;  
\item[\rm (ii)]
$\sig_{yx}(t) = \sig_{xy}(1-t)$ for $t \in [0,1]$;
\item[\rm (iii)]
$d(\sig_{xy}(t),\sig_{x'y'}(t)) \le (1-t)\,d(x,x') + t\,d(y,y')$ for
$t \in [0,1]$.
\end{enumerate}
(This corresponds to a conical and reversible geodesic bicombing in the 
terminology of~\cite{DesL}.) Notice that these conditions
do not ensure that $t \mapsto d(\sig_{xy}(t),\sig_{x'y'}(t))$ is a 
{\em convex} function on $[0,1]$. This is guaranteed under the following 
extra assumption on the traces:
\begin{enumerate}
\item[\rm (iv)]
$\im(\sig_{pq}) \sub \im(\sig_{xy})$ whenever $p = \sig_{xy}(r)$ and
$q = \sig_{xy}(s)$ with $r \le s$.
\end{enumerate}
Note that then $\sig_{pq}(t) = \sig_{xy}((1-t)r + ts)$ for 
$t \in [0,1]$ by~(i). A bicombing~$\sig$ satisfying~(iv) will be 
called {\em consistent}. Busemann spaces and convex subsets of 
normed spaces possess consistent bicombings, whereas some additional examples 
of (general) bicombings are obtained via $1$-Lipschitz retractions onto 
subspaces. We refer to~\cite{DesL} for more information.
 
Our first main result is the following generalization of the 
hyperbolicity criterion for cocompact $\CAT(0)$ spaces stated on 
p.~119 in~\cite{Gro}. A detailed proof of Gromov's result, 
inspired by~\cite{Ebe}, was given in~\cite{Bri}. 
For the case of Busemann spaces, 
both Theorem~\ref{Thm:intro-flat-plane} and Theorem~\ref{Thm:intro-flat-torus} 
below were shown by Bowditch~\cite{Bow}.  

\begin{Thm}[Flat plane] \label{Thm:intro-flat-plane}
Let $X$ be a proper metric space with a consistent bicombing~$\sig$ and 
with cocompact isometry group. 
Then $X$ is hyperbolic if and only if $X$ does not contain an 
isometrically embedded normed plane.
\end{Thm}

Another well-known result from the theory of spaces of nonpositive curvature
is the Flat Torus Theorem, originally proved for smooth manifolds 
in~\cite{GroW,LawY} (see also~\cite{Pre,EelS} for some earlier contributions 
in this direction). A detailed account of this result and its applications 
in the context of $\CAT(0)$ spaces is given in Chap.~II.7 of~\cite{BriH}. 
We have:

\begin{Thm}[Flat torus] \label{Thm:intro-flat-torus}
Let $X$ be a proper metric space with a consistent bicombing~$\sig$.
Let\/ $\Gam$ be a group acting properly and cocompactly by isometries on $X$,
and suppose that $\sig$ is\/ $\Gam$-equivariant. 
If\/ $\Gam$ has a free abelian subgroup group $A$ of rank $n \ge 1$, 
then $X$ contains an isometrically embedded $n$-dimensional normed space 
on which $A$ acts by translations.
\end{Thm}

Here, $\sig$ being {\em $\Gam$-equivariant} means that 
$\gam \circ \sig_{xy} = \sig_{\gam(x)\gam(y)}$ for all $\gam \in \Gam$ and 
$(x,y) \in X \times X$. For example, beyond uniquely geodesic spaces,
every injective metric space (or absolute $1$-Lipschitz retract) $X$
admits a bicombing~$\sig$ that is equivariant with respect to the 
full isometry group $\Isom(X)$ of $X$; see Proposition~3.8 in~\cite{Lan}. 
Furthermore, it is shown in~\cite{DesL} that every proper 
metric space $X$ with a bicombing and with finite combinatorial dimension in 
the sense of~\cite{Dre} also admits a {\em unique consistent}\/ bicombing, 
which is $\Isom(X)$-equivariant, too.

We briefly introduce some terminology that will be used throughout the paper.
Let $X$ be a metric space. A map $\xi \colon I \to X$ of some interval 
$I \sub \R$ is a {\em geodesic} if there is a constant $c \ge 0$, 
the {\em speed} of $\xi$, such that $d(\xi(t),\xi(t')) = c|t - t'|$ for all
$t,t' \in I$. A {\em line} or a {\em ray} in $X$ is a unit speed geodesic
defined on $\R$ or $\R_+ := [0,\infty)$, respectively. Two lines $\xi,\xi'$ are 
{\em parallel}\/ if $\sup_{s \in \R} d(\xi(s),\xi'(s)) < \infty$, 
and two rays $\eta,\eta'$ are {\em asymptotic} 
if $\sup_{s \in \R_+} d(\eta(s),\eta'(s)) < \infty$. 
A family of geodesics $\xi_a \colon I_a \to X$ indexed by a set $A$ will 
be called {\em coherent} if $t \mapsto d(\xi_a(\alpha(t)),\xi_b(\beta(t)))$
is a convex function on $[0,1]$ whenever $a,b \in A$ and 
$\alpha \colon [0,1] \to I_a$ and $\beta \colon [0,1] \to I_b$
are affine maps\footnote{$\alpha,\beta$ are neither assumed 
to be surjective nor orientation preserving.}. Notice that if $\sig$ is a 
consistent bicombing on $X$ and $A \sub X \times X$ is any set, then
$\{\sig_{xy}: (x,y) \in A\}$ is a coherent family. 
Given a bicombing $\sig$ on $X$, we shall often write 
$[x,y](t)$ for $\sig_{xy}(t)$ and $[x,y]$ for $\im(\sig_{xy})$ without further 
comment. A set $C \sub X$ will be called {\em $\sig$-convex} if 
$[x,y] \sub C$ whenever $x,y \in C$. The  {\em (closed) $\sig$-convex hull}
of a subset $S \sub X$ is the smallest (closed) $\sig$-convex set 
containing~$S$. 
A line $\xi \colon \R \to X$ will be called a {\em $\sig$-line} if its trace 
is $\sig$-convex; equivalently, 
$[\xi(r),\xi(s)](t) = \xi((1-t)r + ts)$ for all $r,s \in \R$ and 
$t \in [0,1]$. For two $\sig$-lines $\xi,\xi'$ the function 
$s \mapsto d(\xi(s),\xi'(s))$ is convex, hence constant in case $\xi,\xi'$ 
are parallel. 

The paper is organized as follows. 
In Sect.~\ref{Sect:flat-strips} 
we discuss a generalized Flat Strip Theorem. Unlike for Busemann spaces, 
the $\sig$-convex hull of a pair of parallel $\sig$-lines may be ``thick'' 
and the two lines may span different, though pairwise isometric, 
flat (normed) strips. We also give a criterion for the existence of an 
embedded normed half-plane.
This is then used in Sect.~\ref{Sect:flat-planes} for the proof of 
Theorem~\ref{Thm:intro-flat-plane}. A variant of this result for injective 
metric spaces is also shown.
In Sect.~\ref{Sect:semi-simple-isometries} we establish basic properties 
of semi-simple isometries of spaces with bicombings. We employ a barycenter 
map for finite subsets which was introduced in the context of Busemann spaces 
in~\cite{EsSH}.
Sect.~\ref{Sect:sig-axes} addresses the question whether a hyperbolic 
(axial) isometry of a metric space with a consistent bicombing $\sig$
also possesses an axis that is at the same time a $\sig$-line. 
This is false in general but holds true for the hyperbolic 
elements of a group $\Gam$ as in Theorem~\ref{Thm:intro-flat-torus}.
As an auxiliary tool we use a fixed point theorem for nonexpansive mappings
proved originally for Banach spaces in~\cite{GoeK}. 
Finally, Sect.~\ref{Sect:flat-tori} contains the proof 
of Theorem~\ref{Thm:intro-flat-torus}, and we conclude by an example in which
the embedded flat cannot be chosen to be $\sig$-convex.

In a subsequent paper~\cite{Des} by the first author it is shown that a 
proper cocompact metric space $X$ with a (possibly non-consistent) 
bicombing contains an isometric copy of the $n$-dimensional normed space $V$
under the following asymptotic condition, also studied in~\cite{Wen}: 
there exist subsets $S_k \sub X$ and a sequence $0 < R_k \to \infty$ such 
that the rescaled sets $(S_k,R_k^{-1}d)$ converge in the 
Gromov--Hausdorff topology to the unit ball of $V$. 
This generalizes a result of Kleiner for Busemann spaces; 
see Proposition~10.22 and the more comprehensive Theorem~D in~\cite{Kle}. 
Likewise, it follows that a proper cocompact metric space $X$ with a 
bicombing contains a flat (normed) $n$-plane whenever there is a 
quasi-isometric embedding of $\R^n$ into $X$, a result which was first 
shown for manifolds of nonpositive curvature in~\cite{AndS}. 
Using these more recent results from~\cite{Des} one can extend 
Theorem~\ref{Thm:intro-flat-plane} and Theorem~\ref{Thm:intro-flat-torus}, 
except possibly for the fact that the subgroup $A$ acts on the embedded flat, 
to the case of general bicombings. It is not clear, however, whether this
yields significant improvements. In fact, it is still an open problem whether 
there exists a metric space, proper or not, that admits a bicombing but no 
consistent bicombing. In any case, the arguments presented here are more 
direct. If the (consistent) bicombing~$\sig$ in 
Theorem~\ref{Thm:intro-flat-plane} is equivariant with respect to a cocompact 
group of isometries of $X$, and $X$ is non-hyperbolic, then the construction 
we describe produces an embedded normed plane that 
is foliated by mutually parallel $\sig$-lines in at least one direction.


\section{Flat strips and half-planes} \label{Sect:flat-strips}

We start with the construction of an embedded flat strip in an arbitrary metric 
space, using only a minimal coherent family of geodesics, as defined in the 
introduction. 

\begin{Prop}[Flat strip] \label{Prop:flat-strip}
Let $X$ be a metric space.  
Suppose that $\{\xi,\xi'\} \cup \{\eta_s: s\in \R\}$  
is a coherent collection of geodesics in $X$, where $\xi,\xi'$ are two 
parallel lines with disjoint images and $\eta_s \colon [0,1] \to X$ 
is a geodesic from $\xi(s)$ to $\xi'(s)$ for every $s \in \R$.
Then the map
\[
f \colon \R \times [0,1] \to X, \quad f(s,t) = \eta_s(t),
\]
is an isometric embedding with respect to the metric on 
$\R \times [0,1]$ induced by some norm on $\R^2$.
\end{Prop}

\begin{proof}
For $r \in \R$, put $\nu(r) := d(\xi(0),\xi'(r))$. 
We have $d(\xi(R),\xi'(R+r)) = \nu(r)$ for every $R\in\R$ since the left 
hand side is a convex non-negative bounded function of $R$, hence constant. 
We claim that for every pair of points $p = (s,t)$ and 
$p' = (s+\Del s,t+\Del t)$ in $\R\times [0,1]$ we have
\begin{equation} \label{eq:nu}
d(f(p),f(p')) =
\begin{cases} 
| \Del t | \, \nu\bigl(\tfrac{\Del s}{\Del t}\bigr) 
& \text{if $\Del t \neq 0$,} \\
| \Del s | 
& \text{if $\Del t = 0$.}
\end{cases}
\end{equation}
There is no loss of generality in assuming $\Del t \ge 0$.
Suppose first that $\Del t > 0$, and put $r := \frac{\Del s}{\Del t}$.
Let $q := (s-tr,0)$ and $q' := (s+(1-t)r,1)$ denote the points
where the line through $p,p'$ intersects $\R \times \{0\}$ and 
$\R \times \{1\}$. Then 
\begin{equation} \label{eq:qq'}
d(f(q),f(q')) = d(\xi(s-tr),\xi'(s-tr+r)) = \nu(r).
\end{equation}
We put $\eta := \eta_s$ and $\eta' := \eta_{s+\Del s}$.
By convexity, we get that 
\begin{equation}
d(f(q),f(p)) = d(\xi(s-tr),\eta(t)) 
\le t\, d(\xi(s-r),\eta(1)) = t\, \nu(r).
\end{equation}
Likewise, we have
\begin{equation}
d(f(p'),f(q')) \le (1-t-\Del t)\, \nu(r)
\end{equation}
as well as $d(\eta(0),\eta'(\Del t)) \le \Del t\, \nu(r)$
and $d(\eta(1-\Del t),\eta'(1)) \le \Del t\, \nu(r)$. Hence, by the convexity
of $\lam \mapsto d(\eta(\lam),\eta'(\lam + \Del t))$ on $[0,1-\Del t]$,
also
\begin{equation} \label{eq:pp'}
d(f(p),f(p')) = d(\eta(t),\eta'(t + \Del t)) \le \Del t\, \nu(r).
\end{equation}
From \eqref{eq:qq'}--\eqref{eq:pp'} and the triangle inequality
we see that all inequalities derived so far are in fact 
equalities. In view of~\eqref{eq:pp'}, this shows in particular the 
first part of~\eqref{eq:nu}. The second case follows by continuity from 
the first, since $\bigl| |r| - \nu(0) \bigr| \le \nu(r) \le \nu(0) + |r|$ 
for all $r \in \R$ and hence
\[
\lim_{\Del t \to 0} |\Del t|\, \nu\bigl(\tfrac{\Del s}{\Del t}\bigr) = |\Del s|. 
\]
Now, to conclude the proof, note that $\nu(r) > 0$ for all $r \in \R$, 
as $\xi$ and $\xi'$ have disjoint images.
It then follows readily from~\eqref{eq:nu} that there is a norm $\|\cdot\|$
on $\R^2$ such that $d(f(p),f(p')) = \|p' - p\|$ for all 
$p,p' \in \R \times [0,1]$. Note that the triangle inequality for 
$\|\cdot\|$ is just inherited from $X$.
\end{proof}

The following example shows that, in general, if we replace $\xi$ by 
$s \mapsto \xi(s + a)$ for some $a \ne 0$, we may get a different strip 
in $X$.

\begin{Expl} \label{Expl:diff-strips}
Define piecewise affine functions $g,\b g \colon \R \times [0,1] \to \R$ 
such that 
\[
g(s,t) = \begin{cases}
\tfrac12t & \text{if $s \le 0$,}\\
\tfrac12|s-t| & \text{if $0 \le s \le 1$,}\\
\tfrac12(1-t) & \text{if $s \ge 1$,}
\end{cases}
\]
and $\b g(s,t) = \frac12 - g(s,1-t)$. Note that $g = \b g$ outside of 
$(0,1)^2$, whereas the graphs of $g$ and $\b g$ over $[0,1]^2$
bound a simplex in $\R^3$. 
Consider the sets $Y = \R \times [0,1] \times \R$ and   
\[
X := \{(s,t,u) \in Y: g(s,t) \le u \le \b g(s,t)\},
\]
both equipped with the metric induced by the maximum norm on $\R^3$.
Let $\b\sig \colon (x,y,\lam) \mapsto (1-\lam)x + \lam y$ be the 
canonical bicombing on $Y$. The vertical retraction 
$\pi \colon (s,t,u) \mapsto (s,t,\min\{\max\{u,g(s,t)\},\b g(s,t)\})$ 
from $Y$ onto $X$ is $1$-Lipschitz.
It follows from results in~\cite{DesL} (see Lemma~2.1, Theorem~1.1, 
the observation at the end of Sect.~2, and Theorem~1.2) that 
$\tilde\sig := \pi \circ \b\sig|_{X \times X \times [0,1]}$ is a bicombing on $X$
and that $X$ possesses a unique consistent bicombing~$\sig$, which
furthermore satisfies $\sig_{xy} = \tilde\sig_{xy}$ whenever the consistency 
condition~(iv) stated in the introduction holds for $\tilde\sig_{xy}$. 
In particular, the geodesics $\xi,\xi' \colon \R \to X$ defined by 
$\xi(s) := (s,0,g(s,0))$ and $\xi'(s) := (s,1,g(s,1))$ are two (parallel)
$\sig$-lines. It is then not difficult to see that the strip formed
by the segments $\sig_{\xi(s)\xi'(s+1)}$ corresponds to the graph of $g$, 
whereas the segments $\sig_{\xi(s+1)\xi'(s)}$ trace out the graph of $\b g$. 
\end{Expl}

We also see that in Proposition~\ref{Prop:flat-strip}, 
for fixed $t \in (0,1)$, the lines $s \mapsto f(s,t)$ need not be 
$\sig$-lines in general: 
clearly the lines $s \mapsto \sig_{\xi(s)\xi'(s+1)}(\frac12)$ and 
$s \mapsto \sig_{\xi(s+1)\xi'(s)}(\frac12)$ in the above example cannot both be 
$\sig$-lines. However, it is not difficult to deduce the following result.

\begin{Thm}[Flat strips] \label{Thm:flat-strips}
Let $X$ be a metric space with a consistent bicombing $\sig$, and let
$\xi,\xi' \colon \R \to X$ be two parallel $\sig$-lines with disjoint traces.
Then there exists a norm on $\R^2$ such that the following assertions hold
for the metric it induces on $\R \times [0,1]$:
\begin{enumerate}
\item[\rm (1)]
For every $a \in \R$, the map $f_a \colon \R \times [0,1] \to X$ satisfying
$f_a(s+ta,t) = \sig_{\xi(s)\xi'(s+a)}(t)$ for all $(s,t) \in \R \times [0,1]$ 
is an isometric embedding.
\item[\rm (2)]
If, in addition, $X$ is proper, there also 
exists an isometric embedding $f \colon \R \times [0,1] \to X$ such that 
$f(\cdot,0) = \xi$, $f(\cdot,1) = \xi'$, and $s \mapsto f(s,t)$ is a 
$\sig$-line parallel to $\xi$ and $\xi'$ for every fixed $t \in (0,1)$.
\end{enumerate}   
\end{Thm}

For a corresponding (but easier) result in the case of Busemann spaces, 
see Lemma~1.1 and the remark thereafter in~\cite{Bow} (compare also 
Proposition~5.3 in~\cite{FoeLS}). 
Part~(1) is a direct consequence of Proposition~\ref{Prop:flat-strip}, and 
(2) then follows by a simple limit argument (notice that in~(1),  
all $f_a$ satisfy $f_a(\cdot,0) = \xi$ and $f_a(\cdot,1) = \xi'$).
As this result will not be used in the sequel, we omit the details.

We now proceed to an existence result for embedded flat half-planes, which 
will be instrumental in the proof of Theorem~\ref{Thm:intro-flat-plane}. 
We need the following analogue of the Tits cone in the case of 
$\CAT(0)$ spaces.   
Let $\cR$ be a coherent collection of rays in $X$, and denote by 
$\cR(\infty)$ the set of equivalence classes of mutually asymptotic rays
in $\cR$. For $(a,\xi),(b,\eta) \in \R_+ \times \cR$, we put
\[
d_\infty((a,\xi),(b,\eta)) 
:= \lim_{\lam \to \infty} \frac1\lam\, d(\xi(a\lam),\eta(b\lam)).
\]
Note that the limit exists by convexity,
and $|a - b| \le d_\infty((a,\xi),(b,\eta)) \le a + b$. 
This defines a pseudometric $d_\infty$ on $\R_+ \times \cR$, and the respective 
quotient metric space is a metric cone over $\cR(\infty)$
(compare~\cite{Bal}, p.~38). In particular, for $a > 0$, 
$d_\infty((a,\xi),(a,\eta)) = a\, d_\infty((1,\xi),(1,\eta))$, and this is zero 
if and only $\xi$ and $\eta$ are asymptotic.
The following result should now be compared with Proposition~II.4.2 
in~\cite{Bal} and Proposition~II.9.8 and Corollary~II.9.9 in~\cite{BriH}.

\begin{Prop} [Flat half-plane] \label{Prop:flat-half-plane}
Let $X$ be a metric space. Suppose that $\{\xi\} \cup \{\eta_s: s \in \R\}$ 
is a coherent collection of geodesics in $X$, where $\xi$
is a line and every $\eta_s$ is a ray with 
$\eta_s(0) = \xi(s)$ asymptotic to $\eta := \eta_0$. 
Then, for all $a,b > 0$, the function $s \mapsto d(\xi(s+a),\eta_s(b))$ 
is non-decreasing on $\R$ with limit
\begin{equation} \label{eq:lim-lim}
\lim_{s \to \infty} d(\xi(s+a),\eta_s(b)) = d_\infty((a,\xi),(b,\eta)).
\end{equation}
Furthermore, if for every $a \in \R$ the function $s \mapsto
d(\xi(s+a),\eta_s(1))$ is constant on $\R$ and nonzero, then the map
\[
f \colon \R \times \R_+ \to X, \quad f(s,t) := \eta_s(t),
\]
is an isometric embedding with respect to the metric on $\R \times \R_+$ 
induced by some norm on $\R^2$.
\end{Prop}

\begin{proof}
Let $a,b > 0$. 
First we show that for all $0 < r \le \lam \le r+1$,
\begin{equation} \label{eq:r-lam}
d(\xi(ar + a),\eta_{ar}(b)) \ge \frac1\lam\, d(\xi(a\lam),\eta(b\lam)).
\end{equation}
Since $\eta_{ar}$ and the ray $t \mapsto \eta(br + t)$ are 
asymptotic, we have 
\begin{equation} \label{eq:eta-eta}
d(\eta_{ar}(b),\eta(br + b)) \le d(\eta_{ar}(0),\eta(br)) 
= d(\xi(ar),\eta(br)).
\end{equation}
It follows that 
\begin{align*}
d(\xi(ar + a),\eta_{ar}(b)) 
&\ge d(\xi(ar + a),\eta(br + b)) - d(\xi(ar),\eta(br)) \\
&\ge \biggr(\frac{r+1}{\lam} - \frac{r}{\lam}\biggl) d(\xi(a\lam),\eta(b\lam)),
\end{align*}
which is~\eqref{eq:r-lam}. Putting $\lam = 1$ we get 
$d(\xi(ra + a),\eta_{ra}(b)) \ge d(\xi(a),\eta_0(b))$. 
Likewise, for all $s \in \R$ and $0 < r \le 1$,
\[
d(\xi(s + ra + a),\eta_{s + ra}(b)) \ge d(\xi(s + a),\eta_s(b)),
\]
so $s \mapsto d(\xi(s+a),\eta_s(b))$ is non-decreasing. 
Furthermore, for all $s \in \R$ and $\lam \ge 1$, we have
\[
d(\xi(s + a),\eta_s(b)) 
\le \frac1\lam\, d(\xi(s + a\lam),\eta_s(b\lam)).
\]
Together with~\eqref{eq:r-lam}, this gives~\eqref{eq:lim-lim}.

For the second part of the proposition we have that 
$s \mapsto d(\xi(s+a),\eta_s(1))$ is constant for every $a\in\R$ and 
that these values $\nu(a) := d(\xi(a),\eta_0(1))$ are all positive. 
We first claim that 
\begin{equation} \label{eq:t-nu}
d(\xi(s+ta),\eta_s(t)) = t\,\nu(a)
\end{equation}
for all $t\geq 0$. The left hand side is convex as a function 
of $t$, thus it suffices to show this equality for $0 \le t \in \Z$. 
For $t=0,1$, \eqref{eq:t-nu}~clearly holds. Consequently, by convexity, 
$d(\xi(s+ta),\eta_s(t))\geq t\,\nu(a)$ for all $t > 1$. The reverse inequality
for $1 < t \in \Z$ follows by the triangle inequality since
\[
d(\eta_{s+ka}(t-k),\eta_{s+ka-a}(t-k+1)) \le 
d(\eta_{s+ka}(0),\eta_{s+ka-a}(1)) = \nu(a)
\]
for $k = t,t-1,\dots,1$ (compare~\eqref{eq:eta-eta}). Next, we claim that
for every pair of points $p = (s,t)$ and $p' = (s + \Del s,t + \Del t)$ 
in $\R\times \R_+$ we have
\begin{equation*}
d(f(p),f(p')) = 
\begin{cases} 
| \Del t | \, \nu\bigl(-\tfrac{\Del s}{\Del t}\bigr) 
& \text{if $\Del t \ne 0$,} \\
| \Del s | 
& \text{if $\Del t = 0$,}
\end{cases}
\end{equation*} 
similarly as in the proof of Proposition~\ref{Prop:flat-strip}.
To show this, suppose without loss of generality that $\Del t \ge 0$. 
Let first $\Del t > 0$, and put $a := \frac{\Del s}{\Del t}$ and 
$q := (s - ta,0)$. Then~\eqref{eq:t-nu} yields
\begin{align*}
d(f(p),f(p')) 
&= d(\eta_s(t),\eta_{s + \Del s}(t + \Del t)) \\
&\le d(\eta_s(0),\eta_{s + \Del s}(\Del t)) \\
&= \Del t\, \nu(-a)
\end{align*}
as well as
$d(f(q),f(p)) = t\, \nu(-a)$ and $d(f(q),f(p')) = (t + \Del t)\, \nu(-a)$.
This gives $d(f(p),f(p')) = \Del t\, \nu(-a)$, as claimed. 
The rest of the proof follows as in Proposition~\ref{Prop:flat-strip}.
\end{proof}


\section{Flat Planes} \label{Sect:flat-planes}

We now turn to Theorem~\ref{Thm:intro-flat-plane}.
Recall that a metric space $X$ is {\em $\del$-hyperbolic}, for some constant 
$\del \ge 0$, if for every quadruple $(w,x,y,z) \in X^4$, 
\[
d(w,y) + d(x,z) \le \max\{d(w,x) + d(y,z), d(w,z) + d(x,y)\} + 2\del. 
\]
If such a $\del$ exists, $X$ is said to be {\em hyperbolic}. 
As is well known, for a geodesic metric space this is equivalent to saying 
that geodesic triangles are slim, in an appropriate sense. 
It also suffices to consider triangles whose sides are given by a fixed 
bicombing:

\begin{Lem} \label{Lem:hyp}
Let $X$ be a metric space with a map that selects for every pair of points 
$x,y \in X$ a geodesic segment $[x,y] = [y,x] \sub X$ connecting them.
If for every triple $(x,y,z) \in X^3$ the segment $[x,z]$ is contained in 
the closed $\frac{\del}{2}$-neighborhood of $[x,y] \cup [y,z]$, 
then $X$ is $\del$-hyperbolic.
\end{Lem}

\begin{proof}
Let $(w,x,y,z) \in X^4$.
The union of the closed $\frac{\del}{2}$-neighborhoods of 
$[x,y]$ and $[y,z]$ contains $[x,z]$. Likewise, the closed 
$\frac{\del}{2}$-neighborhoods of $[z,w]$ and $[w,x]$ cover $[x,z]$. 
It follows that there is either a pair of points 
$x' \in [x,y]$ and $z' \in [z,w]$ with $d(x',z') \le \del$ or a pair of points
$y' \in [y,z]$ and $w' \in [w,x]$ with $d(y',w') \le \del$. In the first case,
\begin{align*}
d(w,y) + d(x,z) 
&\le \bigl(d(w,z') + \del + d(x',y)\bigr) 
+ \bigl(d(x,x') + \del + d(z',z)\bigr) \\
&= d(w,z) + d(x,y) + 2\del.
\end{align*}
Similarly, in the second case, $d(w,y) + d(x,z) \le d(w,x) + d(y,z) + 2\del$.
\end{proof}

In particular, a non-hyperbolic $X$ with a bicombing $\sig$ contains a 
sequence of fatter and fatter $\sig$-triangles. The following argument
then uses a ruled surface construction together with the cocompact isometric 
action to produce a collection of mutually asymptotic rays as in 
Proposition~\ref{Prop:flat-half-plane}. This differs from the strategy 
in~\cite{Bow} and is inspired by the proof for $\CAT(0)$ spaces 
in~\cite{Bri,BriH}, although we make no use of angles.

\begin{proof}[Proof of Theorem~\ref{Thm:intro-flat-plane}]
Let $X$ be a proper, cocompact metric space with a consistent 
bicombing~$\sig$. If $X$ contains an isometrically embedded normed plane,
then clearly $X$ cannot be hyperbolic.

Suppose now that $X$ is not hyperbolic. We show that then $X$ must contain 
an embedded normed plane. We continue to write $[x,y]$ in place of 
$\im(\sig_{xy})$. By Lemma~\ref{Lem:hyp} there are sequences of points 
$y^1_n,y^2_n,y^3_n \in X$ and $p_n \in [y^1_n,y^3_n]$ such that 
\begin{equation} \label{eq:n-ball}
B(p_n,n) \cap ([y^1_n,y^2_n] \cup [y^2_n,y^3_n]) = \es
\end{equation} 
for all integers $n \ge 1$, where $B(p_n,n)$ denotes the closed ball at $p_n$
of radius $n$. Put $r_n(\cdot) := d(p_n,\cdot)$. 
For $i = 1,2$, let $x^i_n$ be a point in $[y^i_n,y^{i+1}_n]$ with minimal 
distance to $p_n$, and let $\xi^i_n \colon [0,r_n(x^i_n)] \to X$ be a 
unit speed parametrization of the segment $[p_n,x^i_n]$ from $p_n$ to $x^i_n$.
Then, for every pair $(i,j) \in \{(1,1),(1,2),(2,2),(2,3)\}$, we define 
the ``ruled surface''
\[
\Del^{\!i,j}_n \colon [0,r_n(x^i_n)] \times [0,r_n(y^j_n)] \to X
\]
so that $\Del^{\!i,j}_n(\cdot,0) = \xi^i_n$ and, for each $s \in [0,r_n(x^i_n)]$, 
$\Del^{\!i,j}_n(s,\cdot)$ is a constant speed parametrization of the segment
$[\xi^i_n(s),y^j_n]$ from $\xi^i_n(s)$ to $y^j_n$. Thus 
\[
d(\Del^{\!i,j}_n(s,t),\Del^{\!i,j}_n(s,t')) 
= \frac{d(\xi^i_n(s),y^j_n)}{r_n(y^j_n)} |t-t'| \le 2|t-t'|,
\]
because $d(\xi^i_n(s),y^j_n) \le r_n(x^i_n) + r_n(y^j_n) \le 2\,r_n(y^j_n)$ 
by the choice of $x^i_n$. Note also that by convexity,
\begin{equation} \label{eq:s-s'}
d(\Del^{\!i,j}_n(s,t),\Del^{\!i,j}_n(s',t)) \le 
d(\xi^i_n(s),\xi^i_n(s')) = |s-s'|.
\end{equation}
It follows that each $\Del^{\!i,j}_n$ is $2$-Lipschitz, where here and 
below we equip $\R^2$ with the $l_1$-metric. 
Furthermore, putting $s' := r_n(x^i_n)$, we notice that
for $0 \le r \le s \le s'$ and $0 \le t \le r_n(y^j_n)$,
\begin{align} \label{eq:s-r}
d(\xi^i_n(r),\Del^{\!i,j}_n(s,t)) 
&\ge r_n(\Del^{\!i,j}_n(s',t)) - r - d(\Del^{\!i,j}_n(s,t),\Del^{\!i,j}_n(s',t)) 
\nonumber \\
&\ge r_n(x^i_n) - r - (s'-s) \nonumber \\
&= s - r 
\end{align}
by the triangle inequality, the choice of $x^i_n$, and~\eqref{eq:s-s'}. 

Now we choose a sequence of isometries $\gam_n$ of $X$ so that 
$\gam_n(p_n) \in K$ for all $n$ and for some fixed compact set $K$. 
By~\eqref{eq:n-ball}, $r_n(x^i_n),r_n(y^j_n) > n$. Since $X$ is proper,
we can extract a sequence $n(k)$ so that each of the four
sequences $\gam_{n(k)} \circ \Del^{\!i,j}_{n(k)}$ converges uniformly on compact 
sets, as $k \to \infty$, to a $2$-Lipschitz map
\[
f^{i,j} \colon \R_+ \times \R_+ \to X
\]
with boundary rays $\xi^i := f^{i,j}(\cdot,0)$ and 
$\eta^j := f^{i,j}(0,\cdot)$. Furthermore, for every $s \in \R_+$,
$\eta^{i,j}_s := f^{i,j}(s,\cdot)$ is a ray asymptotic to $\eta^j$,
so $f^{i,j}$ is in fact $1$-Lipschitz. 
Clearly $\{\xi^i\} \cup \{\eta^{i,j}_s : s \in \R_+\}$ is a coherent
collection of rays. From the construction we also have that 
$d(\eta^1(t),\eta^3(t)) = 2t$ for all $t \ge 0$, in particular
$\eta^1,\eta^3$ are non-asymptotic.   
Hence, there is at least one pair $(i,j)$ such that $\xi^i,\eta^j$ 
are non-asymptotic. We put $f := f^{i,j}$, $\xi := \xi^i$, $\eta := \eta^j$,
and $\eta_s := \eta_s^{i,j}$ for some such pair. 
We claim that for all $a \in \R$ and $b > 0$, the limit
\[
L(a,b) := \lim_{s\to\infty} d(\xi(s+a),\eta_s(b)) 
\]
exists and is strictly positive. Clearly $L(0,b) = b$. If $a > 0$, then
$L(a,b) = d_\infty((a,\xi),(b,\eta)) > 0$ by the first part of 
Proposition~\ref{Prop:flat-half-plane} and since $\xi,\eta$ are non-asymptotic.
If $a < 0$, the same result still shows that 
$s \mapsto d(\xi(s+a),\eta_s(b))$ is non-increasing on $[|a|,\infty)$,
so the limit exists, and $L(a,b) \ge |a|$ as a consequence of~\eqref{eq:s-r}.

Next, for every integer $l \ge 1$, we define the $1$-Lipschitz map
\[
f_l \colon [-l,\infty) \times \R_+ \to X, \quad 
f_l(s,t) := f(l+s,t) = \eta_{l+s}(t).
\]
Then we choose isometries $\b\gam_l$ of $X$ so that 
$(\b\gam_l \circ f_l)(0,0) \in K$ for all $l$ and for some fixed compact 
set $K$. As above, there exists a subsequence $l(k)$ such that the sequence 
$\b\gam_{l(k)} \circ f_{l(k)}$ converges uniformly on compact 
sets to a $1$-Lipschitz map
\[
\b f \colon \R \times \R_+ \to X
\]
with boundary line $\b\xi := \b f(\cdot,0)$ and mutually asymptotic
rays $\b\eta_s := \b f(s,\cdot)$ for $s \in \R$. Again, 
$\{\b\xi\} \cup \{\b\eta_s : s \in \R\}$ is a coherent collection of 
geodesics. For every $a \in \R$ and $b > 0$, we now
have that 
\[
d(\b\xi(s+a),\b\eta_s(b)) = L(a,b) > 0
\] 
for all $s \in \R$. 
Hence, by Proposition~\ref{Prop:flat-half-plane},
$\b f$ is an isometric embedding with respect to some norm on $\R^2$.
Using once more that $X$ is cocompact, we then conclude that $X$ contains an 
isometrically embedded normed plane. 
\end{proof}

It is clear that if $X$ is a $\CAT(0)$ or a Busemann space, then this property 
is inherited by any isometrically embedded normed plane, thus the corresponding
norm must be Euclidean or strictly convex, respectively. We briefly discuss
another variant of Theorem~\ref{Thm:intro-flat-plane}, which happens 
to have a very short proof, without reference to bicombings. Recall that 
a metric space $X$ is {\em injective} (as an object in the metric category with 
$1$-Lipschitz maps as morphisms), if for every metric space $B$ and every 
$1$-Lipschitz map $f \colon A \to X$ defined on a set $A \sub B$ there is 
$1$-Lipschitz extension $\olf \colon B \to X$. By a remarkable result
of Isbell~\cite{Isb}, every metric space $Y$ has an injective hull $\E(Y)$,
thus every isometric embedding of $Y$ into an injective 
metric space $X$ factors as $Y \sub \E(Y) \to X$ (see Sects.~2 and~3 
in~\cite{Lan} for a survey). 

Let now $Q = \{w,x,y,z\}$ be any metric space of cardinality four, and suppose 
that $c := d(w,y) + d(x,z)$ is not less than the maximum of
$a := d(w,x) + d(y,z)$ and $b := d(w,z) + d(x,y)$. The injective hull 
(or the {\em tight span}~\cite{Dre}) of $Q$ is isometric to the 
(possibly degenerate) rectangle
$[0,\frac12(c-a)] \times [0,\frac12(c-b)]$ in $(\R^2,\|\cdot\|_1)$ with 
four segments of appropriate lengths attached at the corners, where the 
terminal points of these segments correspond to $Q$. (See Fig.~A1 on
p.~336 in~\cite{Dre}. It is also worth pointing out that the $1$-skeleton 
of the tight span of $Q$, viewed as polyhedral complex, is the unique optimal
network realizing the metric of $Q$; see p.~325 in the same paper.)
Now the $\del$-hyperbolicity of $Q$ means precisely that the width 
(the minimum of the two side lengths) of this
$l_1$-rectangle is not bigger than~$\del$. This has the following easy
consequence.
 
\begin{Thm} 
A proper, cocompact injective metric space $X$ is hyperbolic 
if and only if $X$ does not contain an isometric copy of 
$(\R^2,\|\cdot\|_1)$ or, equivalently, of $(\R^2,\|\cdot\|_\infty)$.
\end{Thm}

\begin{proof}
Suppose that $X$ is not hyperbolic. Then, by the above observation, for 
arbitrarily large $\del > 0$ there exists a quadruple $Q \sub X$ 
whose injective hull contains an isometric copy of 
$[0,\del] \times [0,\del] \sub (\R^2,\|\cdot\|_1)$. Since $X$ is injective,
this $l_1$-square embeds isometrically into $X$ by the respective property
of the injective hull. From a sequence of such squares with side lengths 
tending to infinity we obtain an isometric embedding of the entire $l_1$-plane,
using the fact that $X$ is proper and cocompact.
\end{proof}


\section{Semi-simple isometries} \label{Sect:semi-simple-isometries}

In preparation for the proof of Theorem~\ref{Thm:intro-flat-torus}
we now discuss semi-simple isometries of a metric space $X$ with a 
(not necessarily consistent) bicombing~$\sig$.
The main purpose is to establish basic properties regarding 
sets of minimal displacement, analogous to those in the case
of $\CAT(0)$ spaces. Whereas in the latter case a key role is played by 
the projection onto convex subspaces, we use a barycenter map
for finite subsets of $X$ instead, which we first describe.
The same tool will be employed again in Sect.~\ref{Sect:flat-tori}.

In~\cite{EsSH}, Es-Sahib and Heinich introduced an elegant 
barycenter construction for Busemann spaces, which was reviewed and partly 
improved in a recent paper by Navas~\cite{Nav}. The construction and proofs 
translate almost verbatim to spaces with bicombings. 
For finite subsets, the result is as follows.

\begin{Thm} \label{Thm:barycenter}
Let $X$ be a complete metric space with a bicombing~$\sig$. 
For every integer $n \ge 1$ there exists a map $\bary_n\colon X^n\to X$ 
such that
\begin{enumerate}
\item[\rm (1)] 
$\bary_n(x_1,\ldots,x_n)$ lies in the closed $\sig$-convex hull of 
$\{x_1,\ldots,x_n\}$;
\item[\rm (2)] 
$d(\bary_n(x_1,\ldots,x_n),\bary_n(y_1,\ldots,y_n)) \le 
\min_{\pi \in S_n} \frac1n\sum_{i=1}^n d(x_i,y_{\pi(i)})$;
\item[\rm (3)] 
$\gam\bary_n(x_1,\ldots,x_n) = \bary_n(\gam x_1,\ldots,\gam x_n)$ whenever 
$\gam$ is an isometry of $X$ and $\sig$ is $\gam$-equivariant.
\end{enumerate}
\end{Thm}

We shall sometimes suppress the subscript $n$.
The construction is such that $\bary_1(x) := x$, 
$\bary_2(x,y) := \sig_{xy}(\frac12) = \sig_{yx}(\frac12)$, and, for $n \ge 3$,
\[
\bary_n(x_1,\ldots,x_n) = \bary_n(\bary_{n-1}(\bx^1),\ldots,\bary_{n-1}(\bx^n)),
\]
where $\bx^i := (x_1,\ldots,x_{i-1},x_{i+1},\ldots,x_n)$. The proof of 
Theorem~\ref{Thm:barycenter} is then not very difficult. The more profound 
observation of~\cite{EsSH,Nav} is that the above construction can be 
modified so as to yield a barycenter map on the space
of probability measures with finite first moment that is $1$-Lipschitz  
with respect to the $1$-Wasserstein metric. We will not use this more 
elaborate construction in the present paper.

Now we turn to the discussion of isometries.
We begin by recalling some standard terminology and basic facts. 
First, let $X$ be an arbitrary metric space.
For any map $\gam \colon X \to X$ we denote 
by $d_\gam(x) := d(x,\gam(x))$ the {\em displacement} at a point $x \in X$, 
and we put
\[
|\gam|:=\inf_{x\in X} d_\gam(x) \quad \text{and} \quad
\Min(\gam) := \{x \in X: d_\gam(x) = |\gam|\}.
\]
An isometry $\gam$ of $X$ is called {\em parabolic} if $\Min(\gam)$ is empty
and {\em semi-simple} otherwise. In the latter case, $\gam$ is {\em elliptic} 
if $|\gam| = 0$ (that is, $\gam$ has a fixed point) 
and {\em hyperbolic} if $|\gam| > 0$. 

For an isometry $\gam$, a line $\xi \colon \R \to X$ will be called an
{\em axis} of $\gam$ if there exists a $t > 0$ such that
\begin{equation} \label{eq:axis}
\gam(\xi(s)) = \xi(s + t) \quad \text{for all $s \in \R$.}
\end{equation} 
Then, for $x := \xi(0)$ and any $y \in X$, the triangle inequality gives
\begin{align} \label{eq:gam-n}
d(x,\gam^n(x))
&\le d(x,y) + n\, d_\gam(y) + d(\gam^n(y),\gam^n(x)) \nonumber \\
&= n\, d_\gam(y) + 2\, d(x,y)
\end{align}
for all $n \ge 1$, where $d(x,\gam^n(x)) = nt$, thus 
$t \le d_\gam(y)$ and so $d_\gam(x) = t = |\gam|$.
Hence every isometry $\gam$ with an axis is hyperbolic, and all
axes of $\gam$ are contained in $\Min(\gam)$. 

For the converse, let $\gam$ be a hyperbolic isometry of $X$ with 
$|\gam| =: t$, let $x \in \Min(\gam)$, and suppose there is a 
geodesic $\tau \colon [0,t] \to X$ from $x$ to $\gam(x)$. Then the curve
$\xi \colon \R \to X$ satisfying
\begin{equation} \label{eq:tau}
\xi(nt + s) = \gam^n(\tau(s)) \quad
\text{for all $n \in \Z$ and $s \in [0,t]$}
\end{equation}
is a local geodesic (in fact it preserves all distances less than or equal 
to $t$), because $\xi$ is parametrized by arc length and 
$d(\xi(nt + s),\xi(nt + t + s)) = d_\gam(\tau(s)) \ge t$.
This curve $\xi$ also satisfies~\eqref{eq:axis}, hence it is an axis of
$\gam$ if it happens to be a line. This is the case, for example, 
if $X$ is a Busemann space, as then every local geodesic in $X$ is a geodesic.
(If $\eta \colon [a,b] \to X$ is the geodesic from $\xi(a)$
to $\xi(b)$, then the nonnegative function $s \mapsto d(\xi(s),\eta(s))$
is locally convex, hence convex on $[a,b]$, hence identically zero as
it vanishes at the endpoints.) Thus every hyperbolic isometry of a Busemann
space is axial (compare Chap.~11 in~\cite{Pap}).

The following result shows in particular that this last fact remains true 
in the more general context of this paper. Recall that a bicombing $\sig$ 
of $X$ is {\em $\gam$-equivariant}, for an isometry $\gam$ of $X$, if 
$\gam \circ \sig_{xy} = \sig_{\gam(x)\gam(y)}$ for all $(x,y) \in X^2$.

\begin{Prop} \label{Prop:gam-axis}
Let $\gam \colon X \to X$ be an isometry of a complete metric space $X$ with 
a $\gam$-equivariant bicombing $\sig$. Then:
\begin{enumerate}
\item[\rm (1)]
For all $x,y \in X$ and $n \ge 1$, 
$|\gam| \le \frac1n\, d(x,\gam^nx) \le d_\gam(y) + \frac2n\, d(x,y)$.
\item[\rm (2)]
For all $x \in X$, $\lim_{n\to\infty}\frac1n\, d(x,\gam^nx) = |\gam|$.
\item[\rm (3)]
If $\gam$ is hyperbolic, then for every $x \in \Min(\gam)$ there 
exists an axis of $\gam$ through $x$.
\item[\rm (4)]
If\/ $C \sub X$ is non-empty, $\sig$-convex, and $\gam$-invariant, 
then $|\gam| = \bigl| \gam|_C \bigr|$.
\end{enumerate}
\end{Prop}

\begin{proof}
The second inequality in~(1) is just~\eqref{eq:gam-n}. 
For the first, we employ the barycenter construction stated above.
Given $x \in X$ and $n \ge 1$, put $\bx := (x,\gam x,\dots,\gam^{n-1}x)$ 
and $\gam\bx := (\gam x,\dots,\gam^nx)$. Then
\begin{align*}
|\gam| \le d(\bary_n\bx,\gam\bary_n\bx) 
= d(\bary_n\bx,\bary_n\gam\bx) 
\le \frac{1}{n}\, d(x,\gam^nx),
\end{align*}
where the last two steps use parts~(3) and~(2) of Theorem~\ref{Thm:barycenter},
respectively.
The limit formula~(2) is an immediate consequence of~(1).
As for~(3), let $\gam$ be a hyperbolic isometry with $|\gam| =: t$, 
and let $x \in \Min(\gam)$. Then, for $y = x$, (1)~shows that 
$d(x,\gam^nx) = nt$ for all $n \ge 1$, thus any curve 
$\xi$ as in~\eqref{eq:tau} is a line. 
Finally, given any set $C$ as in~(4),
its closure $\olC$ is still $\sig$-convex and $\gam$-invariant, and 
furthermore complete. Now fix any $x \in C$ and apply~(2)
for both $\gam$ and $\gam|_{\olC}$ to conclude
that $|\gam| = \bigl| \gam|_{\olC} \bigr|$. 
Clearly $\bigl| \gam|_{\olC} \bigr| = \bigl| \gam|_C \bigr|$.
\end{proof}

The following standard result will be used several times in the sequel.

\begin{Lem} \label{Lem:semi-simple}
Let $X$ be a proper metric space, let\/ $\Gam$ be a group acting properly 
and cocompactly by isometries on $X$, and let 
$\alpha_1,\ldots,\alpha_n \in \Gam$. Given a sequence of points in $X$ 
along which the displacement functions $d_{\alpha_1},\ldots,d_{\alpha_n}$ are 
bounded, there exist a subsequence $x_k$ and isometries $\gam_k \in \Gam$
such that $\gam_kx_k$ converges to a point $z \in X$ and, for every
element $\alpha$ of the subgroup $\langle \alpha_1,\ldots,\alpha_n \rangle$,
$\gam_k\alpha\gam_k^{-1} \in \Gam$ is independent of $k$ 
and\/ $\lim_{k \to \infty} d_\alpha(x_k) = d_\alpha(y)$ for all points $y$
in the sequence $\gam_k^{-1}z$.
\end{Lem}

In particular, for any $\alpha \in \Gam$, starting from a minimizing 
sequence for $d_\alpha$ one gets that 
$|\alpha| = \lim_{k \to \infty} d_\alpha(x_k) = d_\alpha(y)$ for some point~$y$.
This shows that $\Gam$ acts by semi-simple isometries 
(compare Proposition~II.6.10 in~\cite{BriH}). 

\begin{proof}
Since the action is cocompact we may assume, by passing to a subsequence $x_k$,
that there exist $\gam_k \in \Gam$ such that $\gam_kx_k$ converges to a point 
$z \in X$. By assumption the sequence 
$d(\gam_kx_k,\gam_k\alpha_1\gam_k^{-1}(\gam_kx_k)) = d_{\alpha_1}(x_k)$ is 
bounded, so $d(z,\gam_k\alpha_1\gam_k^{-1}z)$ is bounded as well.
Hence, because the action of $\Gam$ is proper, we may pass to a further 
subsequence in order to arrange that $\gam_k\alpha_1\gam_k^{-1}$
is equal to the same element $\ol\alpha_1 \in \Gam$ for all $k$. 
Repeating the argument for $\alpha_2,\ldots,\alpha_n$, we arrive at a map
$\alpha_i \mapsto \ol\alpha_i$ which extends to a homomorphism 
$\alpha \to \ol\alpha$ from $\langle \alpha_1,\ldots,\alpha_n \rangle$ 
into $\Gam$ such that $\gam_k\alpha\gam_k^{-1} = \ol\alpha$ for all $k$. 
Now it follows that 
\[
\lim_{k \to \infty} d_\alpha(x_k) 
= \lim_{k \to \infty} d_{\ol\alpha}(\gam_k x_k)
= d_{\ol\alpha}(z)
= d_\alpha(y)
\]
whenever $y = \gam_k^{-1}z$ for some $k$.
\end{proof}

From the above results we obtain a crucial fact for the proof 
of Theorem~\ref{Thm:intro-flat-torus}.

\begin{Prop} \label{Prop:min-a}
Let $X$ be a proper metric space with a bicombing $\sig$.
Let\/ $\Gam$ be a group acting properly and cocompactly by isometries on $X$,
and suppose that $\sig$ is\/ $\Gam$-equivariant. Then for every finitely 
generated abelian subgroup $A$ of\/ $\Gam$ the set
\[ 
\Min(A) := \bigcap_{\alpha\in A} \Min(\alpha) 
\]
is non-empty (and $\sig$-convex, $\alpha$-invariant for every $\alpha \in A$, 
and closed). 
\end{Prop}

\begin{proof}
For an individual $\alpha \in A$ the set $\Min(\alpha)$ is non-empty, 
as noted after Lemma~\ref{Lem:semi-simple}. Now suppose that $B \sub A$ is a 
finite set such that $\Min(B) := \bigcap_{\beta \in B} \Min(\beta) \ne \es$, and 
let $\alpha \in A \sm B$. Note that $\Min(B)$ is $\sig$-convex and furthermore
$\alpha$-invariant, as $\alpha$ commutes with every element of $B$.
Using Proposition~\ref{Prop:gam-axis}(4) we find a sequence $x_k$ such that 
$d_\alpha(x_k) \to \bigl|\alpha|_{\Min(B)}\bigr| = |\alpha|$ and 
$d_\beta(x_k) = |\beta|$ for all $\beta \in B$.
Applying Lemma~\ref{Lem:semi-simple} for the set 
$B \cup \{\alpha\} \sub \Gam$ we get a point $y \in \Min(B \cup \{\alpha\})$.
This shows that $\Min(B) \ne \es$ for every finite set $B \sub A$.
Exhausting $A$ by an increasing sequence of finite subsets we obtain a sequence
$x_k$ in $X$ such that for every $\alpha \in A$, the sequence $d_\alpha(x_k)$ 
is eventually constant with value $|\alpha|$.
Applying Lemma~\ref{Lem:semi-simple} again, for generators 
$\alpha_1,\ldots,\alpha_n$ of $A$, we conclude that $\Min(A)$ is non-empty.
\end{proof}


\section{$\sig$-Axes} \label{Sect:sig-axes}

In Proposition~\ref{Prop:gam-axis} we showed that every hyperbolic 
isometry~$\gam$ of a complete metric space $X$ with a $\gam$-equivariant 
bicombing~$\sig$ is axial. It is natural to ask whether $\gam$ also admits
an axis that is at the same time a $\sig$-line. Such an axis will be called 
a {\em $\sig$-axis}. It turns out that the answer to this question is negative
in general, see Example~\ref{Expl:no-sigma-axis}. However, we shall prove
in Proposition~\ref{Prop:sigma-axes} that any group $\Gam$ satisfying the 
assumptions of Theorem~\ref{Thm:intro-flat-torus} acts by 
``$\sig$-semi-simple'' isometries, that is, every element has either a 
fixed point or a $\sig$-axis.

We start with an auxiliary result which will be useful in the proof of 
Proposition~\ref{Prop:phi-disp}. In~\cite{GoeK}, Goebel and Koter proved 
a fixed point theorem for ``rotative'' nonexpansive mappings in closed convex 
subsets of Banach spaces. The argument can easily be adapted to the present 
context.

\begin{Thm} \label{Thm:GoeK} 
Let\/ $Y$ be a complete metric space with a bicombing $\sig$.
Then every $1$-Lipschitz map $\phi\colon Y\to Y$ for which there 
exist an $n\geq 2$ and $0 \le a < n$ such that
\[
d(y,\phi^n(y)) \leq a\, d(y,\phi(y)) \quad \text{for all $y \in Y$}
\]
has a fixed point. Furthermore, the fixed point set of $\phi$ is a 
$1$-Lipschitz retract of\/ $Y$.
\end{Thm}

\begin{proof} 
Let $\lam\in(0,1)$ (an appropriate value depending only on 
$a,n$ will be determined at the end of the proof). 
For every $x \in Y$, the map sending $y \in Y$ to $[x,\phi(y)](\lam)$ 
is $\lam$-Lipschitz and thus has a unique fixed point $f_\lam(x) \in Y$ 
by Banach's contraction mapping theorem. 
This yields a map $f_\lam \colon Y\to Y$ with the 
property that $f_\lam(x) = [x,\phi(f_\lam(x))](\lam)$ for all $x \in Y$.
By convexity,
\begin{align*}
d(f_\lam(x),f_\lam(y)) &\le (1-\lam)\, d(x,y) 
+ \lam\, d(\phi(f_\lam(x)),\phi(f_\lam(y))) \\
&\le (1-\lam)\, d(x,y) + \lam \,d(f_\lam(x),f_\lam(y)),
\end{align*}
so $f_\lam$ is $1$-Lipschitz. Furthermore, $f_\lam$ has the same fixed points 
as $\phi$, because $[x,\phi(x)](\lam) = x$ if and only if 
$\phi(x) = x$.
By the assumption on $\phi$,
\begin{align} \label{eq:f-lam-1}
d(y,\phi(f_\lam(y))) 
&\le d(y,\phi^n(y)) 
+ d(\phi^n(y),\phi(f_\lam(y))) \nonumber \\
&\le a\, d(y,\phi(y)) + d(\phi^{n-1}(y),f_\lam(y)).
\end{align}
To estimate the last term, note that, again by convexity,
\begin{align*}
d(\phi^m(y),f_\lam(y)) 
&\le (1-\lam)\, d(\phi^m(y),y) + \lam\, d(\phi^m(y),\phi(f_\lam(y))) \\
&\le (1-\lam)m\, d(\phi(y),y) + \lam\, d(\phi^{m-1}(y),f_\lam(y)), 
\end{align*}
for $m \in \{1,\dots,n-1\}$. It follows that 
\begin{equation} \label{eq:f-lam-2}
d(\phi^{n-1}(y),f_\lam(y)) 
\le (1-\lam)c_\lam(n)\, d(\phi(y),y) + \lam^{n-1} d(y,f_\lam(y)), 
\end{equation}
where $c_\lam(n) := (n-1) + (n-2)\lam + \dots + \lam^{n-2}$. Combining
the fact that $d(y,f_\lam(y)) = \lam\, d(y,\phi(f_\lam(y)))$ 
with~\eqref{eq:f-lam-1} and~\eqref{eq:f-lam-2} we get
\[
d(y,f_\lam(y)) 
\le \frac{a + (1-\lam)c_\lam(n)}{1-\lam^n}\, \lam\, d(y,\phi(y)).
\]
Now let $x \in Y$, and put $y := f_\lam(x)$. Then $y = [x,\phi(y)](\lam)$
and hence $\lam\, d(y,\phi(y)) = (1-\lam)\, d(x,y)$,
thus we obtain
\[
d(f_\lam(x),f_\lam^2(x)) 
\le \frac{a + (1-\lam)c_\lam(n)}{1+\lam+\dots+\lam^{n-1}}\, d(x,f_\lam(x)).
\]
Since the factor on the right converges to $\frac{a}{n} < 1$ for $\lam\to 1$, 
there is a $\lam\in(0,1)$ making it strictly less than $1$.
Then, for every $x \in Y$, the sequence $k \mapsto f_\lam^k(x)$ is Cauchy.
Since $f_\lam$ is $1$-Lipschitz, it follows that the limit point $\rho(x)$ 
of this sequence is a fixed point of $f_\lam$, hence a fixed point of $\phi$,
and $\rho$ is a $1$-Lipschitz retraction of $Y$ onto 
the fixed point set of $\phi$.
\end{proof}

Now let $\gam$ be any isometry of a metric space $X$ with a 
bicombing $\sig$. We associate with $\gam$ the map 
\[
\phi_\gam \colon X \to X, \quad 
\phi_\gam(x) = [\gam x,\gam^{-1}x]\bigl(\tfrac12\bigr).
\]
Note that $d(\phi_\gam(x),\phi_\gam(y)) \le \frac12\,d(\gam x,\gam y) + 
\frac12\,d(\gam^{-1}x,\gam^{-1}y) = d(x,y)$, thus $\phi_\gam$ is $1$-Lipschitz.
Our interest in this map comes from the following simple fact.

\begin{Lem} \label{Lem:phi-fixed-pt}
Let $\gam$ be an isometry of a metric space $X$ with a 
$\gam$-equivariant consistent bicombing $\sig$, and let $x \in X$ be such that 
$\gam(x) \ne x$. Then there exists a $\sig$-axis of $\gam$ through $x$ if 
and only if $x$ is a fixed point of the associated map~$\phi = \phi_\gam$. 
\end{Lem}

\begin{proof}
If $\xi \colon X \to \R$ is a $\sig$-axis of $\gam$ through $x$,  
then clearly $\phi(x) = x$. Conversely, suppose that $x$ is a fixed point of 
$\phi$. Put $t := d_\gam(x)$. Let $\tau \colon [0,t] \to X$ be defined 
by $\tau(s) = [x,\gam x](\frac{s}{t})$, and consider the corresponding 
unit speed curve $\xi \colon \R \to X$ satisfying~\eqref{eq:tau}. 
Since $\phi(x) = x$ and $\sig$ is consistent, it follows that $\xi$ is a 
``local $\sig$-line'', in fact every subsegment of length $t$ is $\sig$-convex.
Then, as in the case of Busemann spaces, it follows that
$\xi$ is a (global) $\sig$-line and hence a $\sig$-axis of $x$.
\end{proof}

We now show that the translation length 
$|\phi_\gam| = \inf_{x\in X}d_{\phi_\gam}(x)$ of $\phi_\gam$ is always zero, 
provided the bicombing $\sig$ is $\gam$-equivariant.

\begin{Prop} \label{Prop:phi-disp}
Let $\gam$ be an isometry of a metric space $X$ with a 
$\gam$-equivariant bicombing $\sig$. Then for all $x \in X$ and 
$n \ge 1$,
\[
d(x,\phi_\gam{}^n(x)) \le \sqrt{n}\,d_\gam(x),
\]
and\/ $|\phi_\gam| = 0$.
\end{Prop}

\begin{proof} 
We assume without loss of generality that $X$ is complete.
We write $\phi := \phi_\gam$. Let $x \in X$.
For all $m \in \Z$ and $0 \le n \in \Z$, 
put $x_{n,m} := \phi^n(\gam^m x)$ and $d_{n,m} := d(x,x_{n,m})$.
Note that $\phi$ and $\gam$ commute because $\sig$ is $\gam$-equivariant.
In particular, for $n \ge 1$, we have  
$x_{n,m} = [x_{n-1,m-1},x_{n-1,m+1}](\tfrac12)$ and hence
\[
d_{n,m} \le \frac12 (d_{n-1,m-1} + d_{n-1,m+1}).
\]
By induction on $n$ this yields
\[
d_{n,m} \le 2^{-n} \sum_{i=0}^n \binom{n}{i} d_{0,m-n+2i} \ .
\]
Since $d_{0,m} \leq |m|\,d_\gam(x)$, we obtain
\[
d_{n,0} \le 2^{-n} \sum_{i=0}^n \binom{n}{i} d_{0,2i-n} 
\le 2\, d_\gam(x) \cdot 2^{-n} \sum_{i=0}^n \binom{n}{i} 
\left| i-\frac{n}{2} \right| .
\]
As $2^{-n}\binom{n}{i}$ is the probability mass function of a binomial 
distribution with parameters $n$ and $\tfrac12$ (number of trials and 
probability of success), let $Z$ be a random variable distributed accordingly. 
Recall that the mean and variance are $\mathrm{E}[Z]=\tfrac{n}{2}$, 
$\mathrm{Var}[Z]=\tfrac{n}{4}$, hence
\begin{align*}
\frac{d_{n,0}}{2\, d_\gam(x)} 
\le \mathrm{E}[|Z-\mathrm{E}[Z]|] 
&= \mathrm{E}\left[\sqrt{(Z-\mathrm{E}[Z])^2}\right]\\
&\le \sqrt{\mathrm{E}[(Z-\mathrm{E}[Z])^2]} 
= \sqrt{\mathrm{Var}[Z]} = \frac{\sqrt{n}}{2}
\end{align*}
by Jensen's inequality. Thus $d(x,\phi^n(x)) = d_{n,0} \leq \sqrt{n}\,d_\gam(x)$.

For any $c > |\gam|$, $Y := \{x \in X: d_\gam(x) \leq c\}$ is a non-empty,
complete and $\sig$-convex set with $\gam(Y) = Y$ and, consequently, 
$\phi(Y) \subset Y$. Now if $|\phi|$ was positive, then for some sufficiently 
large $n$ and for some $a < n$ we would have 
$d(x,\phi^n(x)) \leq c\sqrt{n} \le a |\phi| \le a\, d(x,\phi(x))$ 
for all $x \in Y$, and Theorem~\ref{Thm:GoeK} would provide a fixed point
$y = \phi(y)$, in contradiction to $|\phi| > 0$. This shows that
$|\phi| = 0$.
\end{proof}

The following example shows that, in general, the infimum $|\phi_\gam| = 0$ 
need not be attained. The isometry $\gam$ we construct is axial, but has no 
$\sig$-axis.

\begin{Expl} \label{Expl:no-sigma-axis}
Let $X := l_\infty(\Z)$ be the Banach space of bounded functions 
$x \colon \Z \to \R$, with the supremum norm, and consider the 
affine bicombing $(x,y,\lam) \mapsto (1-\lambda)x + \lambda y$
(there is in fact no other bicombing on $X$, see Theorem~1 in~\cite{GaeM}). 
Let $\rho \colon X \to X$ be the shift map satisfying $\rho(x)(k) = x(k-1)$ 
for all $x \in X$ and $k \in \Z$, and let $p \in X$ be defined by 
$p(k) = 1$ for $k \ge 1$ and $p(k) = 0$ otherwise. 
The isometry $\gam \colon X \to X$, $\gam(x) := \rho(x) + p$, 
satisfies $\|\gam^n(0)\|_\infty = n$ for all $n \ge 1$, 
so $|\gam| = d_\gam(0) = 1$ by Proposition~\ref{Prop:gam-axis}. 
Thus $\gam$ is hyperbolic and hence axial. 
The associated map $\phi = \phi_\gam \colon X \to X$ is given by
\[
\phi(x) = \frac12(\rho(x) + \rho^{-1}(x) - z),
\]
where $z := \rho^{-1}(p) - p$ is the indicator function of $0$. 
Now an $x \in X$ with $\phi(x) = x$ would have to fulfil 
$x(0) = \frac12(x(-1) + x(1) - 1)$ as well as 
$x(k) = \frac12(x(k-1)+x(k+1))$ for all $k \ne 0$, and it is easy to see that
no such bounded function $x \colon \Z \to \R$ exists.
\end{Expl}

In contrast to this example, the following holds.

\begin{Prop} \label{Prop:sigma-axes}
Let $X$ be a proper metric space with a consistent bicombing $\sig$.
Let\/ $\Gam$ be a group acting properly and cocompactly by isometries on $X$,
and suppose that $\sig$ is\/ $\Gam$-equivariant. 
Then every isometry $\alpha \in \Gam$ has either a fixed point or 
a $\sig$-axis.
\end{Prop}

\begin{proof} 
Let $\alpha \in \Gam$. In view of Lemma~\ref{Lem:phi-fixed-pt} we just 
need to show that the associated map $\phi := \phi_\alpha$ has a fixed point. 
Let $r > |\alpha|$. Applying Proposition~\ref{Prop:phi-disp} to the complete,
$\sig$-convex and $\alpha$-invariant set 
$X_r := \{ x \in X : d_\alpha(x) \le r \}$, 
we find a sequence of points in $X_r$ along which the displacement function
$d_\phi$ tends to zero. By Lemma~\ref{Lem:semi-simple} there exist a 
subsequence $x_k$ and isometries $\gam_k \in \Gam$ such that 
$\gam_kx_k$ converges to a point $z \in X$ and 
$\gam_k\alpha\gam_k^{-1} =: \ol\alpha \in \Gam$ is constant. 
Put $\ol\phi := \phi_{\ol\alpha}$. Since $\sig$ is $\gam_k$-equivariant, 
we have that for all $y \in X$,
\[
\gam_k \circ [\alpha^{-1}y,\alpha y] 
= [\ol\alpha^{-1}(\gam_ky),\ol\alpha(\gam_ky)], 
\]
hence $d_\phi(y) = d(\gam_ky,(\gam_k \circ \phi)y) = d_{\ol\phi}(\gam_ky)$.
It follows that 
\[
d_\phi(\gam_k^{-1}z) 
= d_{\ol\phi}(z) 
= \lim_{k \to \infty} d_{\ol\phi}(\gam_k x_k) 
= \lim_{k \to \infty} d_\phi(x_k) = 0, 
\]
thus every $\gam_k^{-1}z$ is a fixed point of $\phi$.
\end{proof}


\section{Flat tori} \label{Sect:flat-tori} 

We now prove Theorem~\ref{Thm:intro-flat-torus}.
Thus, in the following, $X$ denotes a proper metric space with  
a consistent bicombing $\sig$, equivariant with respect to a group $\Gam$ 
that acts properly and cocompactly by isometries on $X$, and $A$ is a free 
abelian subgroup of $\Gam$ of rank $n$. We retain the multiplicative notation 
for $A \sub \Gam$, but we fix once and for all an isomorphism 
$\iota \colon (\Z^n,+) \to (A,\cdot)$. For generic points $a,b \in \Z^n$,
the corresponding elements of $A$ will be denoted by $\alpha := \iota(a)$,
$\beta := \iota(b)$ without further comment. We write  
$b_1 = (1,0,\ldots,0,), \ldots, b_n = (0,\ldots,0,1)$ for the canonical
generators of $\Z^n$ and put $\beta_i := \iota(b_i)$.
With this convention, we can state the assertion of 
Theorem~\ref{Thm:intro-flat-torus} as follows: there exist a norm 
$\|\cdot\|$ on $\R^n$ and an isometric embedding 
$f \colon (\R^n,\|\cdot\|) \to X$ such that 
\begin{equation} \label{eq:f-f}
\alpha f(p) = f(p+a) \quad \text{for all $p \in \R^n$ and $a \in \Z^n$.}
\end{equation}
This implies that $d(f(p),\alpha^n f(p)) = \|na\|$ for all $n \ge 1$,
therefore $\|a\|$ must be equal to the translation length~$|\alpha|$
by Proposition~\ref{Prop:gam-axis}(2). We first show that a norm
with this latter property indeed exists. Notice that we already know from
Proposition~\ref{Prop:min-a} that $\Min(A)$ is non-empty.

\begin{Lem} \label{Lem:MinA}
There is a unique norm \mbox{$\|\cdot\|$} on~$\R^n$ 
such that\/ $\|a\| = |\alpha|$ for every $a \in \Z^n$. 
With respect to the metric on $\Z^n$ induced by this norm, 
the map $a \mapsto \alpha x$ is an isometric embedding of\/ $\Z^n$ into $X$ 
for every $x \in \Min(A)$.
\end{Lem}

\begin{proof} 
Define $\|a\| := |\alpha|$ for all $a \in \Z^n$. Then,
for every $x \in \Min(A)$ and $a,b \in \Z^n$,
\[
\|b - a\| = |\alpha^{-1}\beta| = d(x,\alpha^{-1}\beta x) = d(\alpha x,\beta x),
\]
and this is non-zero if $\alpha \ne \beta$, 
for otherwise $\alpha^{-1}\beta$ would have a fixed point and infinite order 
as $A$ is free, in contradiction to the action being proper.
Furthermore, $\|m a\| = |m|\|a\|$ for $m \in \Z$ because
$|\alpha^m| = |m| |\alpha|$ by Proposition~\ref{Prop:gam-axis}.
It follows that \mbox{$\|\cdot\|$} extends uniquely to a norm on $\Q^n$ and 
then also to a norm on~$\R^n$.
\end{proof}

In the following, $\R^n$ (and $\Z^n,\Q^n$) are always equipped with the metric 
induced by this norm $\|\cdot\|$. The next result will constitute the 
last step of the proof of Theorem~\ref{Thm:intro-flat-torus}.

\begin{Prop} \label{Prop:final-step}
Assume that there exists a sequence of\/ $1$-Lipschitz maps 
$f_k \colon \R^n \to X$ such that for all $p \in \R^n$ and 
\mbox{$i \in \{1,\ldots,n\}$},
\[
\lim_{k \to \infty} d(\beta_if_k(p),f_k(p+b_i)) = 0.
\]
Then there is an isometric embedding $f\colon \R^n\to X$ 
satisfying~\eqref{eq:f-f}.
\end{Prop}

\begin{proof} 
First note that the displacement of $\beta_i$ along the sequence $f_k(0)$ 
is bounded by $d(\beta_if_k(0),f_k(b_i)) + d(f_k(b_i),f_k(0))$, where the first 
term goes to zero by assumption and the second is bounded by $\|b_i\|$. 
By Lemma~\ref{Lem:semi-simple} we may assume, after passing to a subsequence,
that there are isometries $\gam_k \in \Gam$ such that $\gam_k f_k(0)$ converges
to a point in $X$ and $\gam_k\alpha\gam_k^{-1} =: \ol\alpha \in \Gam$ 
is constant for every $\alpha \in A$.
By the Arzel\`a--Ascoli theorem we may further assume that the sequence 
of $1$-Lipschitz maps $\gam_k\circ f_k$ converges uniformly on compact 
sets to a $1$-Lipschitz map $h \colon \R^n \to X$.
Now, for all $p \in \R^n$ and $a \in \Z^n$, we have that
\begin{align*}
d(\alpha\gam_k^{-1}h(p),\gam_k^{-1}h(p+a)) 
&= d(\ol\alpha h(p),h(p+a)) \\
&= \lim_{k\to\infty} d(\ol\alpha\gam_kf_k(p),\gam_kf_k(p+a)) \\
&= \lim_{k\to\infty} d(\alpha f_k(p),f_k(p+a)).
\end{align*}
By assumption this last term is zero if $a \in \{b_1,\ldots,b_n\}$.
Thus, for any fixed $k$, the map $f := \gam_k^{-1}\circ h$ 
satisfies $\alpha f(p) = f(p+a)$ for all $p \in \R^n$ and for all generators
$a = b_i$, hence for all $a \in \Z^n$. 
This property then forces the $1$-Lipschitz map $f$ to be isometric on $\Z^n$
because 
\[
\|a\| = |\alpha| \le d(f(p),\alpha f(p)) = d(f(p),f(p+a)) \le \|a\|
\] 
for all $p,a \in \Z^n$. 
Furthermore, every line segment in $\R^n$ connecting 
two points in $\Z^n$ is embedded isometrically. Since the set of all pairs 
of points which lie on a common such segment is dense in $\R^n\times\R^n$, 
we conclude that $f$ is in fact an isometric embedding.
\end{proof}

Now we proceed as follows. First we construct a $1$-Lipschitz
map $g \colon \R^n \to \Min(A)$ that sends every ray $\R_+ a$ with 
$a \in \Z^n \sm \{0\}$ isometrically to a ($\sig$-)ray asymptotic to
a $\sig$-axis of $\alpha$. Then we use a discrete averaging process based on 
barycenters (Theorem~\ref{Thm:barycenter}) to find a sequence of maps 
satisfying the assumptions of Proposition~\ref{Prop:final-step}.

\begin{proof}[Proof of Theorem~\ref{Thm:intro-flat-torus}] 
We fix a point $x \in \Min(A)$ and define a map 
$g \colon \R^n \to X$ as follows. First, for $a \in \Z^n \sm \{0\}$ and 
$\lam \in [0,1]$, put
\[
g(\lam a) := \lim_{k\to\infty} [x,\alpha^k x]\bigl(\tfrac{\lam}{k}\bigr).
\]
The limit exists by Lemma~5.1 in~\cite{DesL} since the orbit 
$\langle\alpha\rangle x$ stays within finite distance of some $\sig$-axis
of $\alpha$ by Proposition~\ref{Prop:sigma-axes},
and the definition is clearly consistent for distinct representations 
of the same point.
For $a,b \in \Z^n \sm \{0\}$ and $\lam \in [0,1]$ we have 
\begin{align*}
d(g(\lam a),g(\lam b)) 
&= \lim_{k\to\infty} d\bigl( [x,\alpha^kx]\bigl(\tfrac{\lam}{k}\bigr), 
[x,\beta^kx]\bigl(\tfrac{\lam}{k}\bigr) \bigr) \\
&\le \lim_{k\to\infty} \frac{\lam}{k}\, d(\alpha^k x, \beta^k x) \\ 
&= \lam \|a - b\|,
\end{align*}
in particular $g$ is $1$-Lipschitz on $\Q^n \sm \{0\}$.
It follows that $g$ extends uniquely to a $1$-Lipschitz map 
$g \colon \R^n \to X$.

Next, for all integers $k \ge 1$, put $I_k := [-k,k]^n \cap \Z^n$. 
We want to establish the following estimate of sublinear growth:
\begin{equation} \label{eq:sublinear}
e(k) := \sup_{a \in I_k} d(g(a),\alpha x) = o(k) 
\quad (k\to\infty).
\end{equation}
Given an \mbox{$\eps>0$}, there is a finite set $B \sub \Z^n$ and 
a constant $C$ such that for every $a \in \Z^n$ there exist a 
point $b \in B$ and a positive integer $m$ such that 
$\|a - mb\| \le \eps \|a\| + C$. 
For each $b \in B$ we pick a point $y_b \in \Min(A)$ on a $\sig$-axis of 
$\beta$. Then 
\[
d(g(mb),\beta^m y_b) 
= \lim_{k \to \infty} d\bigl( [x,\beta^{mk}x]\bigl(\tfrac{1}{k}\bigr),
[y_b,\beta^{mk} y_b]\bigl(\tfrac{1}{k}\bigr) \bigr)  
\le d(x,y_b),
\] 
hence $d(g(mb),\beta^m x) \le 2\, d(x,y_b)$. 
So let $D := \max\{2\, d(x,y_b): b \in B\}$.
Now, for every $a \in \Z^n$, if $b$ and $m$ are as above, we have
\begin{align*}
d(g(a),\alpha x) 
&\leq d(g(a),g(mb)) + d(g(mb),\beta^m x) + d(\beta^m x,\alpha x) \\
&\leq 2\|a - mb\| + D \\
&\leq 2(\eps\|a\| + C) + D.
\end{align*}
This clearly yields~\eqref{eq:sublinear}.

To conclude the proof we now construct maps that meet the requirements 
of Proposition~\ref{Prop:final-step}. Define $f_k \colon \R^n \to X$ by
\[ 
f_k(p) := \bary(\{ \alpha^{-1} g(p+a) : a\in I_k \}) .
\]
Since $g$ is $1$-Lipschitz, it follows from 
Theorem~\ref{Thm:barycenter}(2) that $f_k$ is $1$-Lipschitz as well.
For the generators $b_i$ we have
\begin{align*}
\beta_i f_k(p) &= \bary(\{ \alpha^{-1} g(p+b_i+a) : a \in I_k - b_i\}) , \\
f_k(p+b_i) &= \bary(\{ \alpha^{-1} g(p+b_i+a) : a \in I_k \}) ,
\end{align*}
the first equality being a consequence of Theorem~\ref{Thm:barycenter}(3) 
and a change of variable. In order to estimate $d(\beta_i f_k(p),f_k(p+b_i))$
we need a pairing of points in $I_k$ with points in $I_k-b_i$.
We match $a \in I_k \cap (I_k-b_i)$ with itself and $a \in I_k \sm (I_k-b_i)$ 
with $\tilde a := a - (2k+1)b_i \in (I_k-b_i) \sm I_k$. 
For a pair of the latter type we have
\begin{align*}
d(\alpha^{-1} g(p+b_i+a), x)
&= d(g(p+b_i+a), \alpha x) \\
&\le d(g(p+b_i+a),g(a)) + d(g(a),\alpha x) \\
&\le \|p+b_i\| + e(k)
\end{align*}
as well as $d(\tilde\alpha^{-1}g(p+b_i+\tilde a), x) \le \|p+b_i\| + e(k+1)$,
since $I_k-b_i \sub I_{k+1}$. Thus
\[
d(\alpha^{-1} g(p+b_i+a), \tilde\alpha^{-1} g(p+b_i+\tilde a)) 
\le 2(\|p+b_i\| + e(k+1)) .
\]
Since there are $(2k+1)^{n-1}$ such pairs $(a,\tilde a)$ out of 
$|I_k|=(2k+1)^n$ pairs in total, we conclude that
\[
d(\beta_i f_k(p),f_k(p+b_i)) 
\leq \frac{2(\|p+b_i\| + e(k+1))}{2k+1} \to 0 \quad (k\to\infty)
\]
by Theorem~\ref{Thm:barycenter}(2) and~\eqref{eq:sublinear}.
\end{proof}

We conclude this section with an example illustrating 
Theorem~\ref{Thm:intro-flat-torus}. Given $X$, $\sig$, and an isometric 
embedding $f \colon V \to X$ of some $n$-dimensional normed space $V$ as 
in the theorem, $f$ carries the canonical bicombing $\b\sig$ on $V$ to 
a consistent bicombing on the image of $f$. However, the geodesics 
$f \circ \b\sig_{pq}$ will in general not agree with $\sig_{f(p)f(q)}$. 
In fact, the following example for $n = 2$ shows that, despite of much 
flexibility in the construction of $f$, it may happen that $\im(f)$ is never 
$\sig$-convex. This stands in contrast to the case $n = 1$ treated in 
Proposition~\ref{Prop:sigma-axes}.
 
\begin{Expl} 
Let $w \colon \R \to \R$ be the $1$-periodic function satisfying 
$w(t) = |t|$ for $t \in [-\frac12,\frac12]$. Define piecewise affine 
functions $g,\b g \colon \R^2 \to \R$ by $g(s,t) := w(t)$ and 
$\b g(s,t) := \max\{w(s),w(t)\}$, and consider the set 
\[
X := \{(s,t,u) \in \R^3: g(s,t) \le u \le \b g(s,t)\},
\]
endowed with the metric induced by the maximum norm on $\R^3$.
It follows as in Example~\ref{Expl:diff-strips} that $X$ admits a unique 
consistent bicombing $\sig$ and that the lines $\xi,\xi' \colon \R \to X$
defined by $\xi(s) := (s,0,0)$ and $\xi'(t) := (\frac12,t,\frac12)$ are 
two $\sig$-lines whose traces are contained in the graphs of $g$ and $\b g$, 
respectively. Clearly $\Z^2$ acts properly and cocompactly by isometries 
on $X$ via $((z,z'),x) \mapsto x + (z,z',0)$, and the bicombing $\sig$ 
is $\Z^2$-equivariant.
Theorem~\ref{Thm:intro-flat-torus} now implies that there exist a norm 
$\|\cdot\|$ on $\R^2$ and an isometric embedding 
$f \colon (\R^2,\|\cdot\|) \to X$ such that $f(p) + (z,z',0) = f(p + (z,z'))$
for all $p \in \R^2$ and $(z,z') \in \Z^2$. 
To describe the image of $f$, 
let $\rho \colon X \to (\R^2, \|\cdot\|_\infty)$ denote the $1$-Lipschitz
projection $(s,t,u) \mapsto (s,t)$. Since the third coordinates of any
two points in $X$ differ by at most $\frac12$, it follows that 
$\rho \circ f$ preserves all distances greater than $\frac12$, but this
forces $\rho \circ f$ to be an isometry altogether. Hence $\rho|_{\im(f)}$ 
is an isometry as well, and this implies in turn that $\im(f)$ is the graph
of a $1$-Lipschitz function $h \colon (\R^2, \|\cdot\|_\infty) \to \R$
such that $g \le h \le \b g$ and $h$ is $\Z^2$-periodic. Now
$g = \b g = 0$ on $\Z^2$ and $g = \b g = \frac12$ on 
$\R \times (\frac12 + \Z)$. It follows in particular that the image of $f$ 
contains the sets $\xi(\Z)$ and $\xi'(\frac12 + \Z)$ but cannot contain 
both the points $\xi(\frac12) = (\frac12,0,0)$ and 
$\xi'(0) = (\frac12,0,\frac12)$. Thus, no matter how $f$ is chosen, 
the image of $f$ will not be $\sig$-convex.
\end{Expl}



\bigskip\noindent
D.~Descombes ({\tt dominic.descombes@math.ethz.ch}),\\ 
U.~Lang ({\tt urs.lang@math.ethz.ch}),\\
Department of Mathematics, ETH Zurich, 8092 Zurich, Switzerland 


\end{document}